\documentclass[11pt]{article}

\usepackage[a4paper,margin=30mm]{geometry}
\usepackage{amsmath,amssymb,amsthm}
\usepackage{aliascnt}
\usepackage{microtype}
\usepackage[colorlinks=true,linkcolor=blue,citecolor=blue,urlcolor=blue]{hyperref}
\usepackage[nameinlink,noabbrev]{cleveref}

\crefname{theorem}{theorem}{theorems}
\Crefname{theorem}{Theorem}{Theorems}
\crefname{proposition}{proposition}{propositions}
\Crefname{proposition}{Proposition}{Propositions}
\crefname{lemma}{lemma}{lemmas}
\Crefname{lemma}{Lemma}{Lemmas}
\crefname{corollary}{corollary}{corollaries}
\Crefname{corollary}{Corollary}{Corollaries}
\crefname{remark}{remark}{remarks}
\Crefname{remark}{Remark}{Remarks}

\newcommand{\C}{\mathbb C}
\newcommand{\N}{\mathbb N}
\newcommand{\Her}{\operatorname{Her}}
\newcommand{\Mat}{\operatorname{Mat}}
\newcommand{\Lor}{\mathcal L}
\newcommand{\Pos}{\operatorname{Pos}}
\newcommand{\Sym}{\operatorname{Sym}}
\newcommand{\rank}{\operatorname{rank}}
\newcommand{\spann}{\operatorname{span}}
\newcommand{\intr}{\operatorname{In}}
\newcommand{\ot}{\otimes}
\newcommand{\ol}[1]{\overline{#1}}

\theoremstyle{plain}
\newtheorem{theorem}{Theorem}[section]
\newaliascnt{proposition}{theorem}
\newtheorem{proposition}[proposition]{Proposition}
\aliascntresetthe{proposition}
\newaliascnt{lemma}{theorem}
\newtheorem{lemma}[lemma]{Lemma}
\aliascntresetthe{lemma}
\newaliascnt{corollary}{theorem}
\newtheorem{corollary}[corollary]{Corollary}
\aliascntresetthe{corollary}

\theoremstyle{remark}
\newaliascnt{remark}{theorem}
\newtheorem{remark}[remark]{Remark}
\aliascntresetthe{remark}

\title{Non-Subhomogeneity of Minimal Operator Systems over Positive Semidefinite and Lorentz Cones}
\author{Tim Netzer}
\date{\today}

\begin{document}
\maketitle

\begin{abstract}
The minimal operator systems over $\Mat_k(\C)_+$, $k\geqslant2$, and the
Lorentz cones $\Lor_m$, $m\geqslant4$, are not subhomogeneous.  For the
$2\times2$ cone we construct extreme positive maps of arbitrarily large
output dimension.  Positive retracts give the remaining cases.  Equivalently,
for fixed $k\geqslant2$ and $d$, there exist $s>d$ and an entangled positive
operator on $\C^k\ot\C^s$ such that every compression of the second factor to
$\C^d$ is separable.
\end{abstract}

\paragraph{AI Declaration.}
The results in this paper were produced almost entirely by an artificial
intelligence system (ChatGPT 5.6 Sol).  I posed the question and suggested
possible approaches.  I evaluated the output, clarified and simplified the
arguments, checked them for correctness, and edited the final text.  There is
no settled convention for attributing such work.  I therefore do not claim
ownership and will not submit the paper to a mathematical journal.  Comments
and corrections are welcome at
\href{mailto:tim.netzer@uibk.ac.at}{tim.netzer@uibk.ac.at}.

\section{Introduction}

An embedding of an operator system into $\Mat_d(\C)$ gives a
finite-dimensional realization.  In matrix-convex language it gives a free
spectrahedral description.  Subhomogeneity allows an embedding into
$\Mat_d(A)$, where $A$ is a commutative $C^*$-algebra, but keeps $d$ fixed.
We prove that the minimal operator systems over the complex positive
semidefinite cones $\Mat_k(\C)_+$, $k\geqslant2$, and the Lorentz cones
$\Lor_m$, $m\geqslant4$, do not have this property.

It suffices to treat $\Mat_2(\C)_+$, which is linearly order-isomorphic to
$\Lor_4$.  Duality reduces this case to the construction of extreme positive
maps from $\Mat_2(\C)$ to $\Mat_s(\C)$ whose value at the identity is
invertible, with $s$ unbounded.  We use Woronowicz's symmetric-power
construction with an explicit alternating diagonal operator.  Positive
retracts then give all larger psd and Lorentz cones.

The psd result answers a question from \cite{DN}.  At level $s$, the minimal
system over $\Mat_k(\C)_+$ is the cone of separable positive operators on
$\C^k\ot\C^s$.  Hence, for every fixed $d$, there are $s>d$ and a mixed
entangled operator $X\geqslant0$ such that
$(I_k\ot V)^*X(I_k\ot V)$ is separable for every isometry
$V\colon\C^d\to\C^s$.

For Lorentz cones the corresponding objects are matrix-convex Euclidean
balls.  If $g\geqslant3$, the Arveson boundary of
$\mathcal W^{\max}(\overline{\mathbb B}_g)$ contains irreducible points of
arbitrarily large matrix size; see \cite{EP}.

\section{Preliminaries}

For a finite-dimensional ordered real vector space $(V,P,u)$, the minimal
operator system over $P$ has cones
\[
  P_s^{\min}
  =\left\{\sum_{\nu=1}^N p_\nu\ot b_\nu\mid
  p_\nu\in P,\ b_\nu\in\Mat_s(\C)_+\right\}.
\]
An operator system $T$ is $d$-subhomogeneous if it admits a unital complete
order embedding into $\Mat_d(A)$ for some commutative $C^*$-algebra $A$.

\begin{lemma}
\label{lem:positive-retract}
Let $(V,P,u)$ and $(W,Q,w)$ be ordered spaces.  Suppose there are unital
positive maps
\[
  \alpha\colon V\to W,
  \qquad \beta\colon W\to V,
  \qquad \beta\alpha=\operatorname{id}_V.
\]
Then $\alpha$ induces a complete order embedding
$P^{\min}\to Q^{\min}$.  Consequently, if $P^{\min}$ is not
$d$-subhomogeneous, then neither is $Q^{\min}$.
\end{lemma}

\begin{proof}
The maps $\alpha$ and $\beta$ are completely positive on the minimal
systems.  Since $\beta\alpha=\operatorname{id}_V$, the map $\alpha$ is a
complete order embedding.  A realization of $Q^{\min}$ in $\Mat_d(A)$ would
therefore restrict to a realization of $P^{\min}$.
\end{proof}
Under the hypotheses of \Cref{lem:positive-retract}, we call $(V,P,u)$ a
unital positive retract of $(W,Q,w)$, or simply call $P$ a retract of $Q$.

For $k\geqslant1$, let $\Mat_k(\C)_+\subseteq\Her_k(\C)$ denote the cone of
positive semidefinite matrices, with order unit $I_k$.  For $m\geqslant2$,
let
\[
  \Lor_m=\{(t,x)\in\mathbb R\times\mathbb R^{m-1}\mid
  t\geqslant\|x\|_2\}
\]
be the Lorentz cone of dimension $m$, with order unit $(1,0)$. 

\begin{lemma}\label{lem:basic-retracts}
The cones $\Mat_2(\C)_+$ and $\Lor_4$, with their specified order units,
are unital order isomorphic.  Moreover, for every $k\geqslant2$ and
$m\geqslant4$, the cone $\Mat_2(\C)_+$ is a retract of
$\Mat_k(\C)_+$, and $\Lor_4$ is a retract of $\Lor_m$.
\end{lemma}

\begin{proof}
Define $\chi\colon\mathbb R^4\to\Her_2(\C)$ by
\[
  \chi(t,x_1,x_2,x_3)
  =\begin{pmatrix}
      t+x_3&x_1-ix_2\\
      x_1+ix_2&t-x_3
    \end{pmatrix}
\]
Then $\chi$ is a unital order isomorphism.

For $k\geqslant2$, define
\begin{align*}
  \alpha_k\colon\Her_2(\C)&\to\Her_k(\C),
  &\alpha_k(a)&=a\oplus\frac{\operatorname{tr}(a)}2I_{k-2},\\
  \beta_k\colon\Her_k(\C)&\to\Her_2(\C),
  &\beta_k(b)&=\text{the upper-left $2\times2$ corner of $b$}.
\end{align*}
They are unital and positive, and
$\beta_k\alpha_k=\operatorname{id}_{\Her_2(\C)}$.  For $m\geqslant4$, define
\begin{align*}
  j_m(t,x_1,x_2,x_3)
    &=(t,x_1,x_2,x_3,0,\ldots,0),\\
  q_m(t,x_1,\ldots,x_{m-1})
    &=(t,x_1,x_2,x_3).
\end{align*}
These maps are unital and positive, with
$q_mj_m=\operatorname{id}_{\mathbb R^4}$.
\end{proof}

Thus \Cref{lem:positive-retract,lem:basic-retracts} reduce the main result to
$\Mat_2(\C)_+$.

\begin{theorem}\label{thm:main}
For every $k\geqslant2$ and $m\geqslant4$, neither
$\bigl(\Mat_k(\C)_+\bigr)^{\min}$ nor $\Lor_m^{\min}$ is
$d$-subhomogeneous for any $d\in\N$.
\end{theorem}

\section{Two reductions}

Let $\Phi\colon\Mat_2(\C)\to\Mat_s(\C)$ be linear.  Its Choi matrix
$C_\Phi=\sum_{i,j=0}^1E_{ij}\ot\Phi(E_{ij})$ satisfies
\[
  (\ol x\ot y)^*C_\Phi(\ol x\ot y)
  =y^*\Phi(xx^*)y,
\]
Thus the Jamio\l kowski--Choi correspondence identifies block-positive
matrices with positive maps \cite{Choi,Jamiolkowski}.  Put
$S=\bigl(\Mat_2(\C)_+\bigr)^{\min}$, and let
\[
  S_s^\vee
  =\{C\in\Her_{2s}(\C)\mid
  \operatorname{tr}(CX)\geqslant0\text{ for every }X\in S_s\}
\]
be the dual cone of $S_s$ for the Hilbert--Schmidt pairing.  Then
\[
  S_s^\vee\cong\Pos(\Mat_2,\Mat_s),
\]
where a positive map on $\Her_2(\C)$ is identified with its complex-linear
extension.  By \cite[Theorem~4.1]{DN}, finite-dimensional duality gives
\begin{equation}\label{eq:subhom-duality}
  \begin{aligned}
    S\text{ is }d\text{-subhomogeneous}
    \quad\Longleftrightarrow\quad
    &S^\vee\text{ is generated by }S_d^\vee\\
    &\text{under sums of matrix compressions}.
  \end{aligned}
\end{equation}
The condition on the right means that every
element at every level is a finite sum of compressions of elements of
$S_d^\vee$.  Equivalently, every positive map
$\Phi\colon\Mat_2(\C)\to\Mat_s(\C)$ has a decomposition
\begin{equation}\label{eq:positive-map-decomp}
  \Phi(a)=\sum_\nu v_\nu^*\Psi_\nu(a)v_\nu,
  \qquad \Psi_\nu\colon\Mat_2(\C)\to\Mat_d(\C)\text{ positive},
  \quad v_\nu\in\Mat_{d,s}(\C).
\end{equation}

\begin{lemma}\label{lem:rank-obstruction}
Let $s>d$.  If $\Phi\colon\Mat_2(\C)\to\Mat_s(\C)$ generates an extreme ray of
$\Pos(\Mat_2,\Mat_s)$ and $\Phi(I_2)$ is invertible, then
$S$ is not $d$-subhomogeneous.
\end{lemma}

\begin{proof}
Suppose $S$ is $d$-subhomogeneous.  By \eqref{eq:subhom-duality} and
\eqref{eq:positive-map-decomp}, write $\Phi=\sum_\nu\Theta_\nu$, where
$\Theta_\nu(a)=v_\nu^*\Psi_\nu(a)v_\nu$.  Extremality gives
$\Theta_\nu=\lambda_\nu\Phi$ for every nonzero summand.  On the other hand,
\[
  \rank\Theta_\nu(I_2)\leqslant\rank v_\nu\leqslant d<s
  =\rank\bigl(\lambda_\nu\Phi(I_2)\bigr),
\]
This is impossible.
\end{proof}

Let $H$ be an $s$-dimensional Hilbert space, and let
$\Phi\colon\Mat_2(\C)\to B(H)$ be positive.  Define
\[
  T_\Phi\colon\Mat_2(\C)\ot H\to H,
  \qquad T_\Phi(a\ot h)=\Phi(a)h.
\]
Also define
\[
  N_\Phi
  =\spann\{a\ot h\mid
  a\in\Mat_2(\C)_+,\ \Phi(a)h=0\}.
\]
The map $\Phi$ is irreducible if its commutant consists of the scalar
operators.

\begin{lemma}\label{lem:kernel-span}
If $\Phi$ is unital and irreducible and $\dim N_\Phi=3s$, then it generates
an extreme ray of $\Pos(\Mat_2,B(H))$.
\end{lemma}

\begin{proof}
Since $\Phi$ is unital, $\dim\ker T_\Phi=3s$.  Hence
$N_\Phi=\ker T_\Phi$.  Suppose $\Phi=\Phi_1+\Phi_2$, where the $\Phi_i$ are
positive.  If $a\geqslant0$ and $\Phi(a)h=0$, then $\Phi_i(a)h=0$.  It follows
that $T_{\Phi_i}=R_iT_\Phi$ and $\Phi_i(a)=R_i\Phi(a)$ for some $R_i$.
Moreover, $R_i=\Phi_i(I_2)$ is selfadjoint.  For $a=a^*$, selfadjointness
gives $R_i\Phi(a)=\Phi(a)R_i$.  Irreducibility makes $R_i$ scalar.
\end{proof}

This factorization is the finite-dimensional core of Woronowicz's
nonextendibility method \cite{Woronowicz}.  Product-vector spanning criteria
for optimal entanglement witnesses appear in \cite{Lewenstein}.  A stronger
criterion for exposed positive maps is proved in \cite{CS}.

\section{Explicit symmetric-power construction}

Let $E=\C^2$ have orthonormal basis $(e_0,e_1)$.  Inner products are linear
in the second variable.  Define the coordinate conjugation $\kappa$ and the
linear unitary $J$ by
\[
  \kappa(z_0e_0+z_1e_1)=\ol{z_0}e_0+\ol{z_1}e_1,
  \qquad Je_0=e_1,\quad Je_1=-e_0,
\]
Put $j=J\kappa$.  Then $j$ is antiunitary, $j^2=-I$, and
\begin{equation}\label{eq:j-properties}
  \langle x,j(x)\rangle=0
  \quad(x\in E).
\end{equation}
For $k\geqslant0$, let $\Sym^k(E)\subseteq E^{\otimes k}$ be the symmetric
tensor power.  Thus $\dim\Sym^k(E)=k+1$.  Write $x^k=x^{\otimes k}$.  For
$k\geqslant1$, let
$\mu_k\colon E\ot\Sym^{k-1}(E)\to\Sym^k(E)$ be the restriction of the
orthogonal projection onto $\Sym^k(E)$, and put $\iota_k=\mu_k^*$.  Then
\begin{equation}\label{eq:pure-inclusion}
  \iota_k(x^k)=x\ot x^{k-1}
  \quad(x\in E).
\end{equation}
Fix $r\geqslant1$, set $n=2r+1$, and put
\[
  H=\Sym^n(E),
  \qquad Q=\Sym^{n-1}(E).
\]
Let $h_0,\ldots,h_n$ and $q_0,\ldots,q_{n-1}$ be the normalized Dicke bases
of $H$ and $Q$, indexed by the number of factors equal to $e_1$.  Define
$C_0,C_1\colon H\to Q$ by
\[
  C_0h_k=\sqrt{\frac{n-k}{n}}\,q_k,
  \qquad
  C_1h_k=\sqrt{\frac{k}{n}}\,q_{k-1},
\]
where nonexistent terms are zero.  Thus
\[
  \iota_nh=e_0\ot C_0h+e_1\ot C_1h
  \quad(h\in H).
\]
On $Q$, define the invertible Hermitian operator
\[
  \sigma q_k=s_kq_k,
  \qquad
  s_k=
  \begin{cases}
    1,&k\text{ even},\\
    -\frac13,&k\text{ odd}.
  \end{cases}
\]
Then $\intr(\sigma)=(r+1,r)$, where
$\intr(T)=(n_+(T),n_-(T),n_0(T))$ denotes inertia; for invertible operators
we omit the final zero entry.

Let $F(u\ot v)=v\ot u$ be the flip on $E\ot E$ and set
\[
  \rho
  =(I_E\ot\iota_n)^*(F\ot\sigma)(I_E\ot\iota_n)
  \in B(E\ot H).
\]
Relative to $E\ot H=(e_0\ot H)\oplus(e_1\ot H)$,
\[
  \rho=
  \begin{pmatrix}
    C_0^*\sigma C_0&C_1^*\sigma C_0\\
    C_0^*\sigma C_1&C_1^*\sigma C_1
  \end{pmatrix}.
\]
There are two one-dimensional blocks of value $1$, on
$\C(e_0\ot h_0)$ and $\C(e_1\ot h_n)$.  For $0\leqslant k<n$, its
restriction to
\[
  \spann\{e_0\ot h_{k+1},e_1\ot h_k\}
\]
is
\[
  D_k=\frac1n
  \begin{pmatrix}
    (n-k-1)s_{k+1}&\sqrt{(n-k)(k+1)}\,s_k\\[1mm]
    \sqrt{(n-k)(k+1)}\,s_k&ks_{k-1}
  \end{pmatrix},
\]
where endpoint terms are zero.  If $k$ is even, then
\[
  n^2\det D_k
  =\frac{k(n-k-1)}9-(n-k)(k+1)<0,
\]
and $D_k$ has inertia $(1,1)$.  If $k$ is odd, both diagonal entries are
positive, and
\[
  n^2\det D_k
  =k(n-k-1)-\frac{(n-k)(k+1)}9>0,
\]
because
\[
  \frac{(n-k)(k+1)}{k(n-k-1)}
  =\left(1+\frac1k\right)
   \left(1+\frac1{n-k-1}\right)
  \leqslant4<9.
\]
There are $r+1$ even indices and $r$ odd indices.  Therefore
\begin{equation}\label{eq:rho-inertia}
  \intr(\rho)=(3r+3,r+1),
\end{equation}
so $\rho$ is invertible.  Write $\rho^{-1}=(R_{ij})_{i,j=0}^1$ as a
$2\times2$ block matrix over $B(H)$, and define
\begin{equation}\label{eq:Phi-explicit}
  \Phi(E_{ij})=R_{ji}
  \qquad(0\leqslant i,j\leqslant1).
\end{equation}
Since $\rho^{-1}$ is Hermitian, $\Phi$ is $*$-linear.

For $n=3$ this is Woronowicz's symmetric-power construction
\cite[Section~6]{Woronowicz}, in Dicke coordinates.  The alternating diagonal
choice of $\sigma$ gives the same construction for every odd $n$.

\begin{proposition}
The map $\Phi\colon\Mat_2(\C)\to B(H)$ is positive, and for every
$0\neq x\in E$,
\begin{equation}\label{eq:rank-one-inertia}
  \intr\bigl(\Phi(xx^*)\bigr)=(n,0,1).
\end{equation}
After an invertible range congruence, it is unital.  The normalized map is
irreducible and spans an extreme ray of $\Pos(\Mat_2,B(H))$.
\end{proposition}

\begin{proof}
Let $V_xh=x\ot h$ and $M_x=\rho^{-1}(x\ot H)$.  From
\eqref{eq:Phi-explicit},
\[
  \Phi(xx^*)=V_x^*\rho^{-1}V_x.
\]
The map
\[
  A_x\colon H\to M_x,
  \qquad A_xh=\rho^{-1}V_xh,
\]
is an isomorphism.  Since $\rho^{-1}$ is Hermitian,
\[
  \langle h,\Phi(xx^*)h'\rangle
  =\langle A_xh,\rho A_xh'\rangle.
\]
Thus $\Phi(xx^*)$ is congruent to $\rho|_{M_x}$.  For the Hermitian form
defined by $\rho$, the orthogonal complement of $M_x$ is $x^\perp\ot H$.
For $0\neq z\in E$, put
\[
  C_z=(\langle z,\mathord\cdot\rangle\ot I_Q)\iota_n.
\]
The map $C_z\colon H\to Q$ is onto and has one-dimensional kernel.  For
$z=e_0$ this follows from the Dicke basis.  The general case follows from
the unitary equivariance of $\iota_n$ under the natural actions on the
symmetric powers.  Also,
\[
  \langle z\ot h,\rho(z\ot h')\rangle
  =\langle C_zh,\sigma C_zh'\rangle,
\]
Thus the restricted form on $H$ is represented by $C_z^*\sigma C_z$.
Surjectivity preserves the positive and negative indices of $\sigma$, and
$\ker C_z$ contributes one zero direction.  Therefore
$\intr(\rho|_{z\ot H})=(r+1,r,1)$.  If a nondegenerate Hermitian form has
inertia $(N_+,N_-)$ and its restriction to a subspace has inertia $(a,b,c)$,
then its restriction to the form-orthogonal complement has inertia
$(N_+-a-c,N_--b-c,c)$; see \cite[Corollary~2.7]{Maddocks}.  For Hermitian
forms this follows by applying the real result to the real part.  All three
indices double.  The real and Hermitian orthogonal complements of a complex
subspace coincide, since testing against $s$ and $is$ forces the Hermitian
pairing with $s$ to vanish.  Apply this with \eqref{eq:rho-inertia} and
$z\in x^\perp$.  We obtain
\[
  \intr\bigl(\Phi(xx^*)\bigr)
  =(3r+3-(r+1)-1,\ r+1-r-1,\ 1)
  =(n,0,1).
\]
Since $\Mat_2(\C)_+$ is generated by its rank-one elements, $\Phi$ is
positive.

Set
\[
  k_x=\mu_n\bigl(j(x)\ot\sigma(x^{n-1})\bigr)\in H.
\]
By \eqref{eq:pure-inclusion} and the definition of $\rho$,
\[
  \rho(j(x)\ot x^n)=x\ot k_x.
\]
Since $\rho$ is invertible, $k_x\neq0$.  Equation \eqref{eq:j-properties}
gives $\Phi(xx^*)k_x=0$.
Equation \eqref{eq:rank-one-inertia} now gives
\[
  \ker\Phi(xx^*)=\C k_x.
\]
Let
\[
  N_0=\spann\{(xx^*)\ot k_x\mid x\in E\}
  \subseteq B(E)\ot H.
\]
Define $\operatorname{vec}\colon B(E)\to E\ot E$ by
$\operatorname{vec}(E_{ij})=e_j\ot e_i$.  Then
$\operatorname{vec}(xx^*)=\kappa(x)\ot x$.  The invertible map
\[
  (I_E\ot\rho^{-1})(\operatorname{vec}\ot I_H)
\]
sends the generators of $N_0$ to
\[
  \kappa(x)\ot j(x)\ot x^n
  =\kappa(x)\ot J\kappa(x)\ot x^n.
\]
By mixed polarization,
\[
  \spann\{\kappa(x)^a\ot x^b\mid x\in E\}
  =\Sym^a(E)\ot\Sym^b(E)
  \qquad(a,b\geqslant0).
\]
Indeed, an annihilating functional gives a polynomial which vanishes
identically in the four real coordinates of $x$; all its coefficients
vanish.  Since $J$ and $\sigma$ are invertible and $\mu_n$ is onto, the
cases $(a,b)=(2,n)$ and $(1,n-1)$ give
\[
  \dim N_0=3(n+1),
  \qquad
  \spann\{k_x\mid x\in E\}=H.
\]
The vectors $k_x$ span $H$, and $\dim H>1$.  Choose $x,y$ such that the
lines $\C k_x$ and $\C k_y$ are distinct.  The positive semidefinite
operators $\Phi(xx^*)$ and $\Phi(yy^*)$ have trivial common kernel.  Hence
$\Phi(xx^*+yy^*)>0$.  Since $xx^*+yy^*\leqslant cI_2$ for some $c>0$, it
follows that $R_0=\Phi(I_2)>0$.
Set
\[
  \widehat\Phi(a)
  =R_0^{-1/2}\Phi(a)R_0^{-1/2}.
\]
The map $\widehat\Phi$ is unital and positive.  Its kernel vectors
$\widehat k_x=R_0^{1/2}k_x$ span $H$.  The vectors
$(xx^*)\ot\widehat k_x$ span a $3(n+1)$-dimensional subspace of
$N_{\widehat\Phi}$.  Since
$N_{\widehat\Phi}\subseteq\ker T_{\widehat\Phi}$ and
$\dim\ker T_{\widehat\Phi}=3(n+1)$, equality holds.

Let $P$ be a projection commuting with every $\widehat\Phi(a)$.  It preserves
each kernel line $\C\widehat k_x$.  Thus
$P\widehat k_x=\varepsilon(x)\widehat k_x$ with
$\varepsilon(x)\in\{0,1\}$.  On the unit sphere of $E$, the function
\[
  \varepsilon(x)
  =\frac{\langle\widehat k_x,P\widehat k_x\rangle}
         {\|\widehat k_x\|^2}
\]
is continuous.  The unit sphere of $E$ is connected, so $\varepsilon$ is
constant.  Since the $\widehat k_x$ span $H$, either $P=0$ or $P=I_H$.
The commutant is therefore scalar.  Now \Cref{lem:kernel-span} proves
extremality.
\end{proof}

\begin{corollary}\label{cor:extreme-family}
For every $r\geqslant1$, there is a unital extreme positive map
$\Phi_r\colon\Mat_2(\C)\to\Mat_{2r+2}(\C)$.
\end{corollary}

\section{Proof of the main theorem}

Fix $d\in\N$.  Choose $r\geqslant1$ such that $2r+2>d$.  By
\Cref{cor:extreme-family} and \Cref{lem:rank-obstruction},
$\bigl(\Mat_2(\C)_+\bigr)^{\min}$ is not $d$-subhomogeneous.

The same conclusion holds for $\bigl(\Mat_k(\C)_+\bigr)^{\min}$,
$k\geqslant2$, and for $\Lor_m^{\min}$, $m\geqslant4$, by
\Cref{lem:basic-retracts,lem:positive-retract}.  Since $d$ was arbitrary,
the theorem follows.

\begin{remark}
The Lorentz threshold is sharp.  Since $\Lor_2$ is simplicial, its minimal
operator system is $1$-subhomogeneous.  The system $\Lor_3^{\min}$ has a
$2\times2$ realization and is $2$-subhomogeneous; see
\cite[Section~3.4]{BDN}.
\end{remark}

\end{document}